\documentclass[a4paper,12pt]{article}
\usepackage{}
\usepackage{mathrsfs}
\usepackage{algorithm, algpseudocode}
\usepackage{amsmath,amssymb,amsthm,latexsym,amsfonts}
\usepackage{pstricks}

\usepackage{graphicx}
\usepackage{color}
\usepackage{ifpdf}

\usepackage{caption}

\begin{document}\baselineskip 0.8cm

\newtheorem{lem}{Lemma}[section]
\newtheorem{thm}[lem]{Theorem}
\newtheorem{cor}[lem]{Corollary}
\newtheorem{exa}[lem]{Example}
\newtheorem{con}[lem]{Conjecture}
\newtheorem{rem}[lem]{Remark}
\newtheorem{obs}[lem]{Observation}
\newtheorem{definition}[lem]{Definition}
\newtheorem{prop}[lem]{Proposition}
\theoremstyle{plain}
\newcommand{\D}{\displaystyle}
\newcommand{\DF}[2]{\D\frac{#1}{#2}}

\renewcommand{\figurename}{{\bf Fig}}
\captionsetup{labelfont=bf}

\title{{\Large\bf Some extremal results on the colorful monochromatic
vertex-connectivity of a graph}\footnote{Supported by NSFC No.11371205, ``973" program
No.2013CB834204, and PCSIRT.}}

\author{\small Qingqiong~Cai, Xueliang~Li, Di~Wu\\
\small Center for Combinatorics and LPMC-TJKLC\\
\small Nankai University, Tianjin 300071, China\\
\small cqqnjnu620@163.com; lxl@nankai.edu.cn; wudiol@mail.nankai.edu.cn}
\date{}
\maketitle

\begin{abstract}
A path in a vertex-colored graph is called a
\emph{vertex-monochromatic path} if its internal vertices have the
same color. A vertex-coloring of a graph is a \emph{monochromatic
vertex-connection coloring} (\emph{MVC-coloring} for short), if
there is a vertex-monochromatic path joining any two vertices in the
graph. For a connected graph $G$, the \emph{monochromatic
vertex-connection number}, denoted by $mvc(G)$, is defined to be the
maximum number of colors used in an \emph{MVC-coloring} of $G$.
These concepts of vertex-version are natural generalizations of the
colorful monochromatic connectivity of edge-version, introduced by
Caro and Yuster. In this paper, we mainly investigate the
Erd\H{o}s-Gallai-type problems for the monochromatic
vertex-connection number $mvc(G)$ and completely determine the exact
value. Moreover, the Nordhaus-Gaddum-type inequality for $mvc(G)$ is
also given.



\end{abstract}

\section{Introduction}

All graphs considered in this paper are simple, finite, undirected
and connected. We follow the terminology and notation of Bondy and
Murty \cite{Bondy}. For a graph $G$, we use $V(G)$, $E(G)$, $n(G)$,
$m(G)$, $\Delta(G)$, $\delta(G)$, $deg(u)$ to denote its vertex set,
edge set, the number of vertices, the number of edges, maximum
degree, minimum degree and the degree of vertex $u$, respectively.
For $D\subseteq V(G)$, let $|D|$ be the number of vertices in $D$,
and $G[D]$ the subgraph of $G$ induced by $D$. We use $d(u,v)$ to
denote the distance between two vertices $u$ and $v$ in $G$, and
$diam(G)$ to denote the maximum distance of any two vertices in $G$.
A $\{u,v\}$-path is a path connecting $u$ and $v$. A
$\{u,v\}$-geodesic is a $\{u,v\}$-path of length $d(u,v)$. We write
$u\sim v$ if $u$ is adjacent to $v$, and $u\nsim v$ if $u$ is not
adjacent to $v$.

A path in an edge-colored graph is a \emph{monochromatic path} if
all the edges on the path are colored the same. An edge-coloring of
a graph is a \emph{monochromatical connection coloring}
(\emph{MC-coloring}, for short) if there is a monochromatic path
joining any two vertices in the graph. For a connected graph $G$,
the \emph{monochromatical connection number}, denoted by $mc(G)$, is
defined to be the maximum number of colors used in an
\emph{MC-coloring} of $G$. An \emph{extremal MC-coloring} is an
MC-coloring that uses $mc(G)$ colors. These concepts were introduced
by Caro and Yuster in \cite{Caro}, where they obtained some
nontrivial lower and upper bounds for $mc(G)$. In \cite{Cai}, we
studied two kinds of Erd\H{o}s-Gallai-type problems for $mc(G)$ and
completely solved them.

As a natural idea, we introduce the vertex-version of these concepts
in the following. A path in a vertex-colored graph is a
\emph{vertex-monochromatic path} if its internal vertices have the
same color. An vertex-coloring of a graph is a \emph{monochromatical
vertex-connection coloring} (\emph{MVC-coloring}, for short), if
there is a vertex-monochromatic path joining any two vertices in the
graph. For a connected graph $G$, the \emph{monochromatical
vertex-connection number}, denoted by $mvc(G)$, is defined to be the
maximum number of colors used in an \emph{MVC-coloring} of $G$. An
\emph{extremal MVC-coloring} is an \emph{MVC-coloring} that uses
$mvc(G)$ colors.

It is worth mentioning that the question for determining the
monochromatic vertex-connection number is a natural opposite
counterpart of the recently well-studied problem of vertex-rainbow
connection number \cite{Krivelevich and Yuster,Lishi,chenli}, where
in the latter we seek to find the minimum number of colors needed in
a vertex-coloring so that there is a vertex-rainbow path joining any
two vertices.

An important property of an extremal \emph{MVC-coloring} is that the
vertices with each color form a connected subgraph. Indeed, if the
subgraph formed by the vertices with a same color is disconnected,
then a new color can be assigned to all the vertices of some
component while still maintaining an \emph{MVC-coloring}. For a
color $c$, the \emph{color subgraph} $G_c$ is the connected subgraph
of $G$ induced by the vertices with color $c$. The color $c$ is
\emph{nontrivial} if $G_c$ has at least two vertices. Otherwise, $c$
is \emph{trivial}. A nontrivial color subgraph with $t$ vertices is
said to \emph{waste} $t-1$ colors.

In this paper, we mainly investigate the Erd\H{o}s-Gallai-type and
Nordhaus-Gaddum-type results for colorful monochromatic
vertex-connectivity of a graph.

The Erd\H{o}s-Gallai-type problem is a kind of extremal problems to
determine the maximum or minimum value of a graph parameter with
some given properties. The interested readers can see the monograph
written by Bollob$\acute{a}$s \cite{Boll}, which has a collection of
such extremal problems in graph theory.

A Nordhaus-Gaddum-type result is a (tight) lower or upper bound on
the sum or product of the values of a parameter for a graph and its
complement. The name ``Nordhaus-Gaddum-type'' is given because
Nordhaus and Gaddum \cite{Nordhaus} first established the type of
inequalities for the chromatic number of graphs in 1956. They proved
that if $G$ and $\overline{G}$ are complementary graphs on $n$
vertices whose chromatic numbers are $\chi(G)$ and
$\chi(\overline{G})$, respectively, then $2\sqrt{n}\leq
\chi(G)+\chi(\overline{G})\leq n+1$. Since then, many analogous
inequalities of other graph parameters have been considered, such as
diameter \cite{Harary 1}, domination number \cite{Harary 2}, rainbow
connection number \cite{X. Li}, and so on \cite{CLL, LiMao}. For a
good survey we refer to \cite{AH}.

The rest of this paper is organized as follows. First, we prove some
upper and lower bounds for $mvc(G)$ in terms of the minimum degree
and the diameter. Then we investigate the Erd\H{o}s-Gallai-type
problem and completely determine the exact value. Finally, the
Nordhaus-Gaddum-type inequality for $mvc(G)$ is given.

\section{Upper and lower bounds for $mvc(G)$}

For a connected graph $G$, we take a spanning tree $T$ of $G$. Color
all the non-leaves in $T$ with one color, and each leave in $T$ with
a distinct fresh color. Clearly, this is an MVC-coloring of $G$ with
$\ell(T)+1$ colors, where $\ell(T)$ is the number of leaves in $T$.
Thus we get the following proposition.
\begin{prop}\label{prop1}
Let $G$ be a connected graph with a spanning tree $T$. Then
$mvc(G)\geq \ell(T)+1\geq 3$.
\end{prop}
In order to obtain a good lower bound for $mvc(G)$,
we need to find a spanning tree with as many leaves as possible.
By the known results about spanning trees with many leaves in
\cite{Caro1, Griggs, Kleitman}, we have
\begin{prop}
Let $G$ be a connected graph on $n$ vertices with minimum degree $\delta$.

$(1)$ If $\delta\geq3$, then $mvc(G)\geq \frac{1}{4}n+3$.

$(2)$ If $\delta\geq4$, then $mvc(G)\geq \frac{2}{5}n+\frac{13}{5}$.

$(3)$ If $\delta\geq5$, then $mvc(G)\geq \frac{1}{2}n+3$.

$(4)$ If $\delta\geq3$, then $mvc(G)\geq
\left(1-\frac{\ln(\delta+1)}{\delta+1}(1+o_{\delta}(1))\right)n+1$.
\end{prop}
We proceed with a lower bound for $mvc(G)$.
\begin{prop}\label{prop2}
Let $G$ be a connected graph with $n$ vertices and diameter $d$.

$(1)$ $mvc(G)=n$ if and if only $d\leq 2$;

$(2)$ If $d\geq 3$, then $mvc(G)\leq n-d+2$ , and the bound is sharp.
\end{prop}
\begin{proof}
(1) holds obviously. For (2),
the vertex-monochromatic path between the two vertices
at distance $d\ (d\geq 3)$ wastes at least $d-2$ colors.
Then $mvc(G)\leq n-d+2$.
For the sharpness, we can take the graph $G_0$
obtained from a copy of $K_{n-d+1}$ by attaching
a path $P$ of length $d-1$ at a vertex $v_0$ in $K_{n-d+1}$.
Clearly, $diam(G_0)=d$.
Give $v_0$ and the internal vertices on $P$ one color,
and each other vertex in $G_0$ a distinct fresh color.
It is easy to check that this vertex-coloring is an MVC-coloring
of $G_0$ using $n-d+2$ colors,
which implies $mvc(G_0)\geq n-d+2$. Thus $mvc(G_0)=n-d+2$.
\end{proof}

\section{Erd\H{o}s-Gallai-type results for $mvc(G)$}

The following problems are called Erd\H{o}s-Gallai-type problems.

\noindent {\bf Problem I:} Given two positive integers $n$, $k$ with
$3\leq k\leq n$, compute the minimum integer $f_v(n,k)$ such that if
a connected graph $G$ satisfies $|V(G)|=n$ and $|E(G)|\geq
f_v(n,k)$, then $mvc(G)\geq k$.

\noindent {\bf Problem II:} Given two positive integers $n$, $k$
with $3\leq k\leq n$, compute the maximum integer $g_v(n,k)$ such
that if a connected graph $G$ satisfies $|V(G)|=n$ and $|E(G)|\leq
g_v(n,k)$, then $mvc(G)\leq k$.

Note that $g_v(n,k)$ does not exist for $3\leq k\leq n-1$,
and $g_v(n,n)=n-1$, since for a star $S_n$ on $n$ vertices, we have $mvc(S_n)=n$.
For this reason, the rest of the section is devoted to studying {\bf Problem I}.

First, we state some lemmas, which are used to determine the value of $f_v(n,k)$.

\begin{lem}\label{lem6}\cite{Ding}
Let $G$ be a connected graph with $|E(G)|\geq |V(G)|+\binom{t}{2}$
and $|V(G)|\neq t+2$. Then $G$ has a spanning tree with at least
$t+1$ leaves, and this is best possible.
\end{lem}

\begin{lem}\label{lem7}\cite{Harary}
The maximum diameter among all connected graphs with $n$ vertices and $m$ edges
is $(n-1)-x(p)+y(p)$, where $p=m-n+1$, $x(p)=\Big\lceil\frac{1+\sqrt{1+8p}}{2}\Big\rceil$,
$y(p)=1$ if $p=\binom{t}{2}$ for some $t$, and $y(p)=2$ otherwise.
\end{lem}

\begin{lem}\label{lem8}
Let $C_n$ be a cycle of order $n$. Then
\begin{align*}
mvc(C_n)=
\begin{cases}
n & n\leq 5\\
3 & n\geq 6
\end{cases}
\end{align*}
\end{lem}
\begin{proof}
For $n\leq 5$, we know $diam(C_n)\leq 2$, and thus $mvc(C_n)=n$. For
$n\in\{6,7\}$, it is easy to check that $mvc(C_n)=3$. For $n\geq 8$,
by Proposition \ref{prop1}, it suffices to prove that $mvc(G)\leq
3$. By contradiction, we assume that $mvc(C_n)\geq 4$. Let $f$ be an
extremal $MVC$-coloring of $C_n$, and $f(v_i)$ the color of vertex
$v_i$. Let $V(C_n)=\{v_1,v_2,\ldots,v_{n}\}$. Since $G$ is
monochromatically vertex-connected, for the pair of antipodal
vertices $\{v_n,v_{\lfloor\frac{n}{2}\rfloor}\}$, there exists a
vertex-monochromatic path $P$ of length at least
$\lfloor\frac{n}{2}\rfloor\geq 4$ connecting them. Without loss of
generality, suppose $P=v_nv_1\ldots v_{\lfloor\frac{n}{2}\rfloor}$.
Then $f(v_1)=f(v_3)$, and we can find three vertices $v_i,\ v_j,\
v_{\ell}\ (\lfloor\frac{n}{2}\rfloor\leq i<j<\ell\leq n)$ with three
different colors but color $f(v_1)$. Then there exist no
vertex-monochromatic paths connecting $v_2$ and $v_j$, a
contradiction.
\end{proof}

\begin{lem}\label{lem9}
Let $G$ be the graph obtained from a complete graph on $\{v_1,\ldots,v_{t+2}\}$ by
replacing the edge $v_{t+1}v_{t+2}$ with a path $P_0=v_{t+2}v_{t+3}\ldots,v_{n}v_{t+1}$.
Then $mvc(G)\leq t+2$ for $1\leq t\leq n-5$.
\end{lem}
\begin{proof}
Suppose that $f$ is an extremal $MVC$-coloring of $G$, and $f(v_i)$
is the color of the vertex $v_i$. Let $V_1=\{v_1,\ldots,v_t\}$, and
$V_2=V(G)\setminus V_1$. Denote by $S$ the set of all pairs
$\{v_{j},v_{\ell}\}$ of vertices in $V_2$ except
$\{v_{t+1},v_{t+2}\}$, such that all the vertex-monochromatic
$\{v_{j},v_{\ell}\}$-paths contain some vertex in $V_1$. We call a
path with color $c$, if all the internal vertices on the path are
colored by $c$.

Case 1: $S=\emptyset$.

Then for each pair $\{v_{j},v_{\ell}\}$ of vertices in $V_2$ except
$\{v_{t+1},v_{t+2}\}$, all the vertex-monochromatic
$\{v_{j},v_{\ell}\}$-paths are contained in $P_0$. Let $v_i$ be any
vertex in $V_1$. For each $v_j\in V_2$, the shortest
vertex-monochromatic $\{v_{i},v_{j}\}$-paths must be
$P=v_iv_{t+1}v_n\cdots v_j$ or $P=v_iv_{t+2}v_{t+3}\cdots v_j$,
which is contained in the cycle $C^i=v_iv_{t+2}\cdots
v_nv_{t+1}v_i$. For $\{v_{t+1},v_{t+2}\}$, $P=v_{t+1}v_iv_{t+2}$ is
a vertex-monochromatic $\{v_{t+1},v_{t+2}\}$-path contained in
$C^i$. Thus $f$ induces an $MVC$-coloring of a cycle
$C^i=v_iv_{t+2}\ldots v_nv_{t+1}v_i$ for each $v_i\in V_1$.

Case 2: $S\neq \emptyset$.

Then for $\{v_j,v_{\ell}\}\in S$ with $j>\ell$, the shortest
vertex-monochromatic $\{v_j,v_{\ell}\}$-paths must be $P=v_j\cdots
v_{t+1}v_iv_{t+2}\cdots v_{\ell}$, where $v_i$ is some vertex in
$V_1$ with $f(v_i)=f(v_{t+1})$ or $f(v_{t+2})$.

Suppose first $f(v_{t+1})=f(v_{t+2})$. Then we can find a vertex
$v_i$ in $V_1$ such that $f(v_i)=f(v_{t+1})=f(v_{t+2})$. Such vertex
$v_i$ must exist; otherwise there are no vertex-monochromatic paths
connecting the pairs of vertices in $S$. For each
$\{v_j,v_{\ell}\}\in S$, $P=v_j\cdots v_{t+1}v_iv_{t+2}\cdots
v_{\ell}$ is a vertex-monochromatic $\{v_j,v_{\ell}\}$-path. With
similar arguments as in Case 1, we get that $f$ induces an
$MVC$-coloring of the cycle $C^i=v_iv_{t+2}\ldots v_nv_{t+1}v_i$.

Now suppose $f(v_{t+1})\neq f(v_{t+2})$, say $f(v_{t+1})=red,
f(v_{t+2})=blue$. Then for $\{v_{j},v_{\ell}\}\in S$, exactly one of
$v_{j},v_{{\ell}}$ must be $v_{t+1}$ or $v_{t+2}$; otherwise, the
vertex-monochromatic $\{v_j,v_{\ell}\}$-paths contain both $v_{t+1}$
and $v_{t+2}$ as internal vertices, but $f(v_{t+1})\neq f(v_{t+2})$,
a contradiction. For $i\in \{1,2\}$, let $S_i$ be the set of pairs
of vertices in $S$ containing $v_{t+i}$. If one of $S_1,S_2$ is
empty, say $S_1\neq \emptyset$ and $S_2= \emptyset$, then we assume
that $\{v_{t+1},v_{\ell}\}\in S_1$ and $P=v_{t+1}v_iv_{t+2}\cdots
v_{\ell}$ is a vertex-monochromatic $\{v_{t+1},v_{\ell}\}$-path,
where $v_{i}\in V_1$. Obviously, $P$ is with color blue. For each
$\{v_{t+1},v_j\}\in S_1(=S)$, $P'=v_{t+1}v_iv_{t+2}\cdots v_{j}$ is
a vertex-monochromatic $\{v_{t+1},v_{j}\}$-path. With similar
arguments as in Case 1, we get that $f$ induces an $MVC$-coloring of
the cycle $C^i=v_iv_{t+2}\ldots v_nv_{t+1}v_i$.

Now consider the case $S_1\neq \emptyset$ and $S_2\neq \emptyset$.
Assume that $\{v_{t+1},v_{{\ell}_1}\}\in S_1$ and $\{v_{t+2},v_{{\ell}_2}\}\in S_2$.
Let $P_1=v_{t+1}v_{i_1}v_{t+2}\cdots v_{{\ell}_1}$
(resp. $P_2=v_{t+2}v_{i_2}v_{t+1}\cdots v_{{\ell}_2}$)
be a vertex-monochromatic path connecting $\{v_{t+1},v_{{\ell}_1}\}$
(resp. $\{v_{t+2},v_{{\ell}_2}\}$),
where $v_{i_1}\in V_1,\ v_{i_2}\in V_1$.
Obviously, $P_1$ is with color blue, while $P_2$ is with color red.
We claim that ${\ell}_2\geq{\ell}_1-1$.
Otherwise, both $P_1$ and $P_2$ contain $v_{\ell_1-1}$ as an internal vertex,
but $P_1$ and $P_2$ are with different colors, a contradiction.
Now we recolor all the vertices in $G$ colored by blue except
$v_{i_1}$ by red, and get a new vertex-coloring $f'$. Next we will
show that $f'$ is still an extremal $MVC$-coloring. It suffices to
consider the pairs of vertices which only have vertex-monochromatic
paths with color blue in $f$. Let $\{x,y\}$ be such a pair, and $P$
a shortest vertex-monochromatic $\{x,y\}$-path with color blue in
$f$. If $P$ does not contain $v_{i_1}$ as an internal vertex, then
$P$ is a vertex-monochromatic $\{x,y\}$-path with color red in $f'$.
Otherwise, $P$ must have the form $(x=)v_{t+1}v_{i_1}v_{t+2}\cdots
v_q(=y)$ ($t+3\leq q\leq n$). Now take the path $P':
(x=)v_{t+1}v_{i_2}v_{t+2}\cdots v_q(=y)$, which is a
vertex-monochromatic $\{x,y\}$-path with color red in $f'$. Thus
$f'$ is an extremal $MVC$-coloring of $G$, in which the vertices
$v_{t+1},v_{t+2}$ receive the same color. This is the case we have
discussed.

Therefore we come to the conclusion that
there exists an extremal $MVC$-coloring of $G$, which
induces an $MVC$-coloring of a cycle $C^i=v_iv_{t+2}\ldots v_nv_{t+1}v_i$
for some $v_i\in V_1$.
Since the cycle $C^i$ has length $n-t+1\geq6$,
we have $mvc(C^i)=3$ by Lemma \ref{lem8}.
So $mvc(G)\leq (t-1)+mvc(C^i)=t+2$.
\end{proof}

\begin{lem}\label{lem10}
Let $G$ be a connected graph with $n$ vertices and $m=\binom{n-2}{2}+2$ edges.
Then $mvc(G)\geq n-1$, and this bound is sharp.
\end{lem}

\begin{proof}
If $\Delta(G)\geq n-2$, then $G$ has a spanning tree $T$ with at
least $n-2$ leaves. Hence $mvc(G)\geq \ell(T)+1\geq n-1$. We are
done. Now we assume $\Delta(G)\leq n-3$. It follows from Lemma
\ref{lem7} that $diam(G)\leq 3$. If $diam(G)=2$, then $mvc(G)=n$ by
Proposition 2.3. We are done. Now we assume $diam(G)=3$. If $G$
contains only one pair $\{u,v\}$ of vertices at distance 3, then
give the two internal vertices of a $\{u,v\}$-geodesic one color,
and each other vertex a different fresh color. Clearly, it is an
$MVC$-coloring of $G$ using $n-1$ colors. Thus $mvc(G)\geq n-1$. We
are done. Now suppose that $G$ contains at least two pairs of
vertices at distance 3. If there exists two pairs
$\{u_1,v_1\},\{u_2,v_2\}$ of vertices at distance 3 such that
$\{u_1,v_1\}\cap\{u_2,v_2\}=\emptyset$, then $u_i,v_i$ are not
adjacent and have no common neighbors, since $d(u_i,v_i)=3$. So we
have $deg(u_i)+deg(v_i)\leq n-2$ for $i\in \{1,2\}$. Thus
$\sum_{v\in V(G)}deg(v)\leq 2(n-2)+(n-4)(n-3)=n^2-5n+8$. On the
other hand, $\sum_{v\in V(G)}deg(v)=2m=n^2-5n+10$, a contradiction.
Now suppose that for any two pairs $\{u_1,v_1\},\{u_2,v_2\}$ of
vertices at distance 3, $\{u_1,v_1\}\cap\{u_2,v_2\}\neq\emptyset$.
We distinguish the following cases.

{\bf Case 1:} All the pairs of vertices at distance 3 have a
common vertex, say $u_1$.

Since $m=\binom{n-2}{2}+2=n-1+\binom{n-3}{2}$, it follows from Lemma
\ref{lem6} that $G$ has a spanning tree $T$ with at least $n-3$
leaves. Hence $mvc(G)\geq n-2$. By contradiction, we assume that
$mvc(G)=n-2$. Let $f$ be an extremal $MVC$-coloring of $G$, and
$f(v_i)$ be the color of vertex $v_i$. Thus $f$ wastes two colors.
This can be classified into the following two subcases:

{\bf Subcase 1.1:} There are two nontrivial colors $R$ and $B$,
and the color subgraph $G_R$ (resp. $G_B$) consists of two adjacent vertices
$w_1, w_2$ (resp. $w_3, w_4$).

Then for each pair $\{u_1,v\}$ of vertices at distance 3,
$\{u_1,v\}$ must be connected by a vertex-monochromatic path with
color $R$ or $B$. Let $X$ be the set of vertices $v$ with
$d(u_1,v)=3$ such that $u_1,v$ can be connected by a
vertex-monochromatic path $P_1$ with color $R$, say $P_1=u_1w_1w_2v$
(this implies $u_1\nsim w_2$). Let $Y$ be the set of vertices $v$
with $d(u_1,v)=3$ such that $u_1,v$ can only be connected by a
vertex-monochromatic path $P_2$ with color $B$, say $P_2=u_1w_3w_4v$
(this implies $u_1\nsim w_4$). See {\bf Fig 1$(1)$}. Clearly, $X\neq
\emptyset$ and $Y\neq \emptyset$; otherwise we can get an
$MVC$-coloring using more colors. Moreover, $\{X,Y\}$ is a partition
of all the vertices at distance 3 from $u_1$.

Let $Z=\{u_1,w_1,w_2,w_3,w_4\}\cup X\cup Y$. For $u\in V(G)\setminus
Z$, if $u\sim u_1$, then $u$ is not adjacent to any vertex in $X\cup
Y$, since the distance between them is 3. If $u\nsim u_1$, then $u$
can not be adjacent to every vertex in $Z\setminus \{u_1\}$;
otherwise we can give $\{u,w_1\}$ one color, and each other vertex a
distinct fresh color, which is an $MVC$-coloring using $(n-1)$
colors. Thus $u$ is not adjacent to at least two vertices in $Z$.
For $v\in X\cup Y$, since $d(u_1,v)=3$, $v$ is not adjacent to
$\{u_1,w_1,w_3\}$. By the definition of $Y$, $w_2$ is not adjacent
to any vertex in $Y$. Furthermore, $w_4$ can not be adjacent to all
the vertices in $X$; otherwise we can give $w_2$ a fresh color, and
get an $MVC$-coloring using $n-1$ colors. As we have noted, $u_1$ is
not adjacent to $w_2$, $w_4$. From the above, we have $m\leq
\binom{n}{2}-2(n-|X|-|Y|-5)-3(|X|+|Y|)-4 =\binom{n}{2}-2n-|X|-|Y|+6
\leq \binom{n}{2}-2n+4<\binom{n-2}{2}+2$, a contradiction.
\begin{figure}[ht]
\begin{center}
\includegraphics[width=5.5cm]{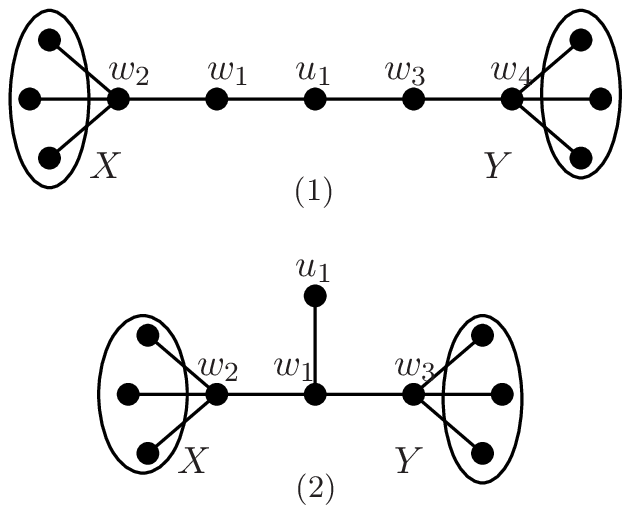}\\
\textbf{Fig 1}: The illustration for Case 1
\end{center}
\end{figure}

{\bf Subcase 1.2:}  There is exactly one nontrivial color $R$, and
the color subgraph $G_R$ consists of three vertices $w_1,w_2,w_3$.

For some pair $\{u_1,v_1\}$ of vertices at distance 3, they are
connected by a vertex-monochromatic path $P_1$ with color $R$.
Without loss of generality, we assume $P_1=u_1w_1w_2v_1$ (this
implies $u_1\nsim w_2$). For $w_3$, there must exist a pair
$\{u_1,v_2\}$ of vertices at distance 3 such that all the
vertex-monochromatic paths $P_2$ connecting them contain $w_3$. If
$P_2=u_1w_3w_1v_2$, then $d(u_1,v_2)=2$, since $u_1\sim w_1$ and
$w_1\sim v_2$, a contradiction. If $P_2=u_1w_3w_2v_2$, then
$P_2'=u_1w_1w_2v_2$ is also a vertex-monochromatic
$\{u_1,v_2\}$-path not containing $w_3$, a contradiction. If
$P_2=u_1w_2w_3v_2$, then $u_1\sim w_2$, a contradiction. Thus $P_2$
must be the the form $P_2=u_1w_1w_3v_2$ (this implies $u_1\nsim
w_3$). Let $X$ be the set of vertices $v$ with $d(u_1,v)=3$ such
that $\{u_1,v\}$ are connected by a vertex-monochromatic path
$P=u_1w_1w_2v$. Let $Y$ be the set of vertices $v$ with $d(u_1,v)=3$
such that $\{u_1,v\}$ can only be connected by a
vertex-monochromatic path $P=u_1w_1w_3v$. See {\bf Fig 1$(2)$}.
Clearly, $X\neq \emptyset$ and $Y\neq \emptyset$. Moreover,
$\{X,Y\}$ is a partition of all the vertices at distance 3 from
$u_1$.

Let $Z=\{u_1,w_1,w_2,w_3\}\cup X\cup Y$.
With similar arguments as in Subcase 1.1, we have
(1) For $u\in V(G)\setminus Z$,
$u$ is not adjacent to at least two vertices in $Z$.
(2) For $v\in X\cup Y$,
$v$ is not adjacent to $u_1,w_1$.
(3) $w_2$ is not adjacent to any vertex in $Y$.
(4) $w_3$ is not adjacent to all the vertices in $X$.
(5) $u_1$ is not adjacent to $w_2$, $w_3$.
From the above, we have
$m\leq \binom{n}{2}-2(n-|X|-|Y|-4)-2(|X|+|Y|)-4
=\binom{n}{2}-2n-4|X|-4|Y|+4
\leq \binom{n}{2}-2n+4<\binom{n-2}{2}+2$,
a contradiction.

Therefore, in Case 1 we have $mvc(G)\geq n-1$.


{\bf Case 2:} There exist three pairs $\{u_i,v_i\}$ $(1\leq i\leq 3)$ of vertices
with $d(u_i,v_i)=3$, such that $\{u_1,v_1\}\cap \{u_2,v_2\}\cap \{u_3,v_3\}=\emptyset$.

Since any two such pairs have a common vertex,
without loss of generality, we may assume
$u_1=u_2,u_3=v_1,v_3=v_2$.
Now the three pairs can be written as
$\{u_1,v_1\},\{u_1,v_2\},\{v_1,v_2\}$.
As two vertices in each pair are at distance 3,
$u_1\nsim v_1$, $u_1\nsim v_2$, $v_1\nsim v_2$,
and each vertex in $V(G)\setminus \{u_1,v_1,v_2\}$
is adjacent to at most one vertex in $\{u_1,v_1,v_2\}$.
Thus $deg(u_1)+deg(v_1)+deg(v_2)\leq n-3$.
Then we have
$\sum_{v\in V(G)}deg(v)\leq n-3+(n-3)(n-3)=n^2-5n+6$.
On the other hand, $\sum_{v\in V(G)}deg(v)=2m=n^2-5n+10$,
a contradiction.

Now we show the sharpness of the bound. Let $G_0$ be the graph
obtained from a complete graph on $\{v_1,\ldots,v_{n-2}\}$. by
adding a path $P_0=v_{n-2}v_{n-1}v_n$ to it. It is easily checked
that $m(G_0)=\binom{n-2}{2}+2$ and $diam(G_0)=3$. By Proposition
2.3, we know $mvc(G_0)\leq n-1$. Hence $mvc(G_0)=n-1$.
\end{proof}

\begin{thm}\label{thm4}
Let $G$ be a connected graph with $n\geq 3$ vertices and $m$ edges.
If $n+\binom{t}{2}\leq m\leq n+\binom{t+1}{2}-1$ for $1\leq t\leq n-2$,
then $mvc(G)\geq t+2$, and this bound is sharp
except for $m=n+\binom{t+1}{2}-1,t\in \{n-3,n-4\}$.
For the latter two cases, $mvc(G)\geq t+3$, and this bound is sharp.
\end{thm}

\begin{proof}
Let $p=m-n+1$. Then $\binom{t}{2}+1\leq p\leq \binom{t+1}{2}$.

{\bf Case 1:} $n=t+2$.

If $\binom{t}{2}+1\leq p\leq \binom{t+1}{2}-1$,
then it follows from Lemma \ref{lem7} that
the diameter of $G$ is at most $n-1-x(p)+y(p)=(t+1)-(t+1)+2=2$.
If $p=\binom{t+1}{2}$, then
the diameter of $G$ is at most $n-1-x(p)+y(p)=(t+1)-(t+1)+1=1$.
By Proposition 2.3, we have $mvc(G)=n=t+2$.

{\bf Case 2:} $n\neq t+2$.

By Lemma \ref{lem6}, we know that $G$ contains a spanning tree $T$
with at least $t+1$ leaves.
Then $mvc(G)\geq \ell(T)+1\geq t+2$.

Next we will show the sharpness of the bound.
If $\binom{t}{2}+1\leq p\leq \binom{t+1}{2}-1$, then
we can take the extremal graph $G_1$ as follows:
First take a complete graph $K_{t+1}$ with vertex set $\{v_1,\ldots,v_{t+1}\}$,
and then add a path $P=v_{t+1},\ldots,v_n$ to it,
and finally add the remaining edges (at most $t-1$)
between $v_{t+2}$ and $\{v_1,\ldots,v_t\}$ randomly.
It is easily checked that $diam(G_1)=n-t$.
By Proposition 2.3, we have $mvc(G_1)\leq t+2$.
Hence $mvc(G_1)=t+2$.
If $p=\binom{t+1}{2}$ and $1\leq t\leq n-5$,
then we can take the extremal graph $G_2$ as in Lemma \ref{lem9}.
It is easily checked that $m(G_2)=n-1+\binom{t+1}{2}$, and
$p=m-n+1=\binom{t+1}{2}$.
By Lemma \ref{lem9}, we have $mvc(G_2)\leq t+2$.
Hence $mvc(G_2)=t+2$.
If $p=\binom{t+1}{2}$ and $t=n-2$,
then $m=\binom{n}{2}$, i.e. $G\cong K_n$.
Thus $mvc(G)=n=t+2$.

If $p=\binom{t+1}{2}$ and $t=n-4$,
then $m=\binom{n-2}{2}+2$.
Now by Lemma \ref{lem10}, we have $mvc(G)\geq n-1=t+3$,
and this bound is sharp.

If $p=\binom{t+1}{2}$ and $t=n-3$,
then it follows from Lemma \ref{lem7} that
the maximum diameter is $n-1-x(p)+y(p)=n-1-(t+1)+1=2$.
Hence $mvc(G)=n=t+3$.
\end{proof}

\begin{cor}\label{cor1}
Given two integers $n,k$ with $3\leq k\leq n$,
\begin{align*}
f_v(n,k)=
\begin{cases}
n-1 & k=3\\
n+\binom{k-2}{2} & 4\leq k \leq n-2\\
n-1+\binom{k-2}{2} & n-1\leq k \leq n\\
\end{cases}
\end{align*}
\end{cor}

\begin{proof}
Since $mvc(G)\geq 3$ for any connected graph $G$, we know $f_v(n,3)=n-1$.
For $4\leq k\leq n-2$,
if $m\geq n+\binom{k-2}{2}$, then it follows from Theorem \ref{thm4}
that $mvc(G)\geq k$.
Hence $f_v(n,k)\leq n+\binom{k-2}{2}$.
For $m=n-1+\binom{k-2}{2}$, by Theorem \ref{thm4},
there exists a graph $G_0$ with $n$ vertices and $m$ edges
such that $mvc(G_0)=k-1$.
Hence $f_v(n,k)\geq n+\binom{k-2}{2}$.
So we get $f_v(n,k)=n+\binom{k-2}{2}$ for $4\leq k\leq n-2$.
For $n-1\leq k \leq n$, if $m\geq n-1+\binom{k-2}{2}$,
then it follows from Theorem \ref{thm4} that
$mvc(G)\geq k$. Hence $f_v(n,k)\leq n-1+\binom{k-2}{2}$.
For $m=n-2+\binom{k-2}{2}$,
by Theorem \ref{thm4}, there exists a graph $G_0$ with $n$ vertices and $m$ edges
such that $mvc(G_0)=k-1$.
Hence $f_v(n,k)\geq n-1+\binom{k-2}{2}$.
So we get $f_v(n,k)=n-1+\binom{k-2}{2}$ for $n-1\leq k \leq n$.
\end{proof}

\section{Nordhaus-Gaddum-type theorem for $mvc(G)$}

A \emph{double star} is a tree with diameter 3.
The \emph{centers} of a double star are the two nonleaves in it.
\begin{lem}\cite{Wu}\label{lem1}
Let $G$ be a connected graph with connected complement
$\overline{G}$. Then

$(1)$ if $diam(G)>3$, then $diam(\overline{G})=2$,

$(2)$ if $diam(G)=3$, then $\overline{G}$
has a spanning subgraph which is a double star.
\end{lem}

As we all know, a connected graph on $n$ vertices has at least $n-1$ edges.
If both $G$ and $\overline{G}$ are connected,
then $2(n-1)\leq e(G)+e(\overline{G})=\binom{n}{2}$,
and so $n\geq 4$.
In the sequel, we always assume that $G$ has at least $n\geq 4$ vertices,
and both $G$ and $\overline{G}$ are connected.
Clearly, for $n=4$, both $G$ and $\overline{G}$ are a path on four vertices.
Thus $mvc(G)=mvc(\overline{G})=3$, and
$mvc(G)+mvc(\overline{G})=6$.
\begin{thm}
If $G$ is a graph on $n\geq 5$ vertices, then $n+3\leq mvc(G)+mvc(\overline{G})\leq 2n$,
and the bounds are sharp.
\end{thm}
\begin{proof}
For any graph $G$, we have a trivial upper bound $mvc(G)\leq n$.
So $mvc(G)+mvc(\overline{G})\leq 2n$.
Now take the graph $G_0$ in {\bf Fig 2}.

\begin{figure}[ht]
\begin{center}
\includegraphics[width=8cm]{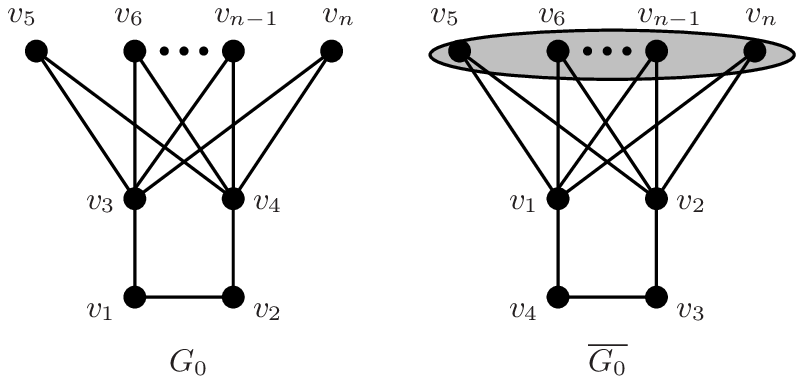}\\
\textbf{Fig 2}: $diam(G_0)=diam(\overline{G_0})=2$
\end{center}
\end{figure}
It is easily checked that $diam(G_0)=diam(\overline{G_0})=2$.
By Proposition \ref{prop2}, we have $mvc(G_0)+mvc(\overline{G_0})=2n$,
which implies the sharpness of the bound.

For the lower bound, if $diam(G)>3$, then by Lemma \ref{lem1}, we
have $diam(\overline{G})=2$. Hence $mvc(G)+mvc(\overline{G})\geq
3+n$. Now we can suppose $diam(G)\leq 3$ and $diam(\overline{G})\leq
3$. If $diam(G)\leq 3$ and $diam(\overline{G})\leq2$, then similarly
we have $mvc(G)+mvc(\overline{G})\geq 3+n$. If
$diam(G)=diam(\overline{G})=3$, then by Lemma \ref{lem1}, $G$ (resp.
$\overline{G}$) contains a double star $S_1$ (resp. $S_2$) as a
spanning subgraph. And $mvc(S_i)\geq n-1$, since we can give the two
centers in $S_i$ one color, and each other vertex a distinct fresh
color, which induces an $MVC$-coloring using $n-1$ colors. Thus
$mvc(G)+mvc(\overline{G})\geq mvc(S_1)+mvc(S_2)\geq 2(n-1)\geq n+3$
for $n\geq 5$. Now we construct a graph $G_0$ that reaches the lower
bound. Just take $G_0=P_{n}$. Since $diam(P_{n})=n-1\geq 4$, it
follows from Lemma \ref{lem1} that $diam(\overline{P_n})=2$. Then
$mvc(P_n)+mvc(\overline{P_n})=3+n$. The proof is complete.
\end{proof}

\end{document}